\documentclass[11pt]{amsart}
\usepackage{graphicx,multirow}
\usepackage[mathcal]{euscript}
\usepackage{float}
\usepackage{amsmath,amsthm,amssymb,amsfonts,amscd,epsfig,latexsym,graphicx,textcomp}

\restylefloat{figure}

\newtheorem{theorem}{Theorem}[section]

\newtheorem{conjecture}[theorem]{Conjecture}

\theoremstyle{definition}

\theoremstyle{remark}
\newtheorem{remark}[theorem]{Remark}

\numberwithin{equation}{section}

\newcommand{\BR}{\mathbb{R}}

\begin{document}

\title[Cube Number Properties]{Cube number can detect chirality and Legendrian type of knots}

\author[B. McCarty]{Ben McCarty}

\address{Department of Mathematics, Louisiana State University \newline
\hspace*{.375in} Baton Rouge, LA 70817, USA} \email{\rm{benm@math.lsu.edu}}

\date{}

\begin{abstract}
For a knot $K$ the cube number is a knot invariant defined to be the smallest $n$ for which there is a cube diagram of size $n$ for $K$.  We will show that the cube number detects chirality in all cases computed thus far, and distinguishes certain legendrian knots.
\end{abstract}

\maketitle

\bigskip
\section{Introduction}
\bigskip

Cube diagrams are $3$-dimensional representations of knots or links (c.f. \cite{Adam}).  A cube diagram is a cubic lattice knot embedded in an $n \times n \times n$ cube in $\mathbb{R}^3$ where each projection to an axis plane ($x = 0$, $y = 0$, and $z = 0$) is a grid diagram.  The integer $n$ is the \emph{size} of the cube diagram and the \emph{cube number} of a knot, denoted $c(K)$, is the smallest $n$ for which there is a cube diagram for the knot of size $n$.  

In \cite{BaldMcCar} small examples of cube diagrams of knots were given up to $7$ crossings.  In some cases it was observed that the examples given were minimal.  In each of these examples minimality was guaranteed since the arc index, $\alpha(K)$, is less than or equal to the cube number and cube diagrams for $K$ can be found with cube number equal to $\alpha(K)$.  Cube number is a far more powerful invariant than arc index.  For example, $\alpha(K) = \alpha(mK)$ where $mK$ is the mirror image of $K$ for all knots and links, but: 

\bigskip
\noindent
{\bf Theorem 1} \emph{For eight of the first twelve chiral knots in Rolfsen's knot table (up to $7$-crossing knots), cube number detects chirality and no counter example is known to exist in the remaining four cases.  See Theorem \ref{thm:chirality} for the full list of knots where cube number is known to detect chirality.}
\bigskip

Let $K$ be a Legendrian knot.  Define the \emph{Legendrian cube number} (or cube number when the context is clear), $c_\ell(K)$, to be the minimum $n$ such that there is a cube diagram for $K$ of size $n$ that projects to a Legendrian front of $K$ (see details in Section \ref{section:legendrian}).  Perhaps surprisingly, the Legendrian cube number detects Legendrian knot type in some cases.  This fact is unexpected because there is no clear relationship as of yet between cube diagrams and Legendrian knots (cube diagrams do not naturally embed as Legendrian knots even when they have the same Legendrian knot projections).  In this paper we prove:

\bigskip
\noindent
{\bf Theorem 2} \emph{Let $K_{min}$ be the left hand $(5,2)$-torus knot with maximal Thurston-Bennequin number and $r = -3$ and $K_{max}$ the $(5,2)$-torus knot with maximal Thurston-Bennequin number and $r = 3$.  Then the Legendrian cube number distinguishes between $K_{min}$ and $K_{max}$.}
\bigskip

We conclude this paper with a conjecture that generalizes Theorem 2 to all $(p,2)$ torus knots and show that this conjecture is true if Conjecture \ref{gridMoves} is true.

\bigskip
\section{Definition of a cube diagram}

Let $n \in \mathbb{Z}^{+}$ and $\Gamma$ an $n \times n \times n$ cube, thought of as a $3$-dimensional Cartesian grid with integer-valued vertices.  A \textit{flat of $\Gamma$} is any cuboid (a right rectangular prism) with integer vertices in $\Gamma$ such that there are two orthogonal edges of length $n$ with the remaining orthogonal edge of length $1$.  A flat with an edge of length 1 that is parallel to the $x$-axis, $y$-axis, or $z$-axis is called an {\em $x$-flat}, {\em $y$-flat}, or {\em $z$-flat} respectively.  Note that the cube itself is canonically oriented by the standard orientation of $\BR^3$ (right hand orientation).

\begin{center}
\begin{figure}[h]
\centering
\includegraphics[scale=.3]{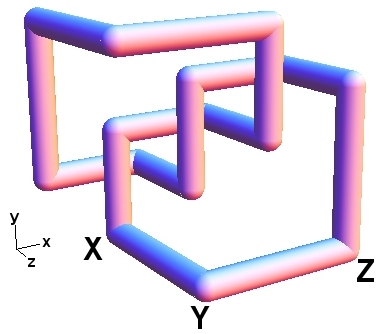}
\caption{Lefthand trefoil cube diagram.}
\label{fig:LHTrefoil}
\end{figure}
\end{center}

A \emph{marking} is a labeled half-integer point in $\Gamma$.  We mark unit cubes of $\Gamma$ with either an $X$, $Y$, or $Z$ such that the following {\em marking conditions} hold:
\begin{itemize}
    \item each flat has exactly one $X$, one $Y$, and one $Z$ marking;\\

    \item the markings in each flat form a right angle such that each segment is parallel to a coordinate axis;\\

    \item for each $x$-flat, $y$-flat, or $z$-flat, the marking that is the vertex of the right angle is an $X, Y,$ or $Z$ marking respectively.
\end{itemize}

We get an oriented link in $\Gamma$ by connecting pairs of markings with a line segment whenever two of their corresponding coordinates are the same.  Each line segment is oriented to go from an $X$ to a $Y$, from a $Y$ to a $Z$, or from a $Z$ to an $X$. The markings in each flat define two perpendicular segments of the link $L$ joined at a vertex, call the union of these segments a {\it cube bend}. If a cube bend is contained in an $x$-flat, we call it an {\it $x$-cube bend}. Similarly, define {\it $y$-cube bends} and {\it $z$-cube bends}.

\begin{center}
\begin{figure}[h]
\centering
\includegraphics[scale=.2]{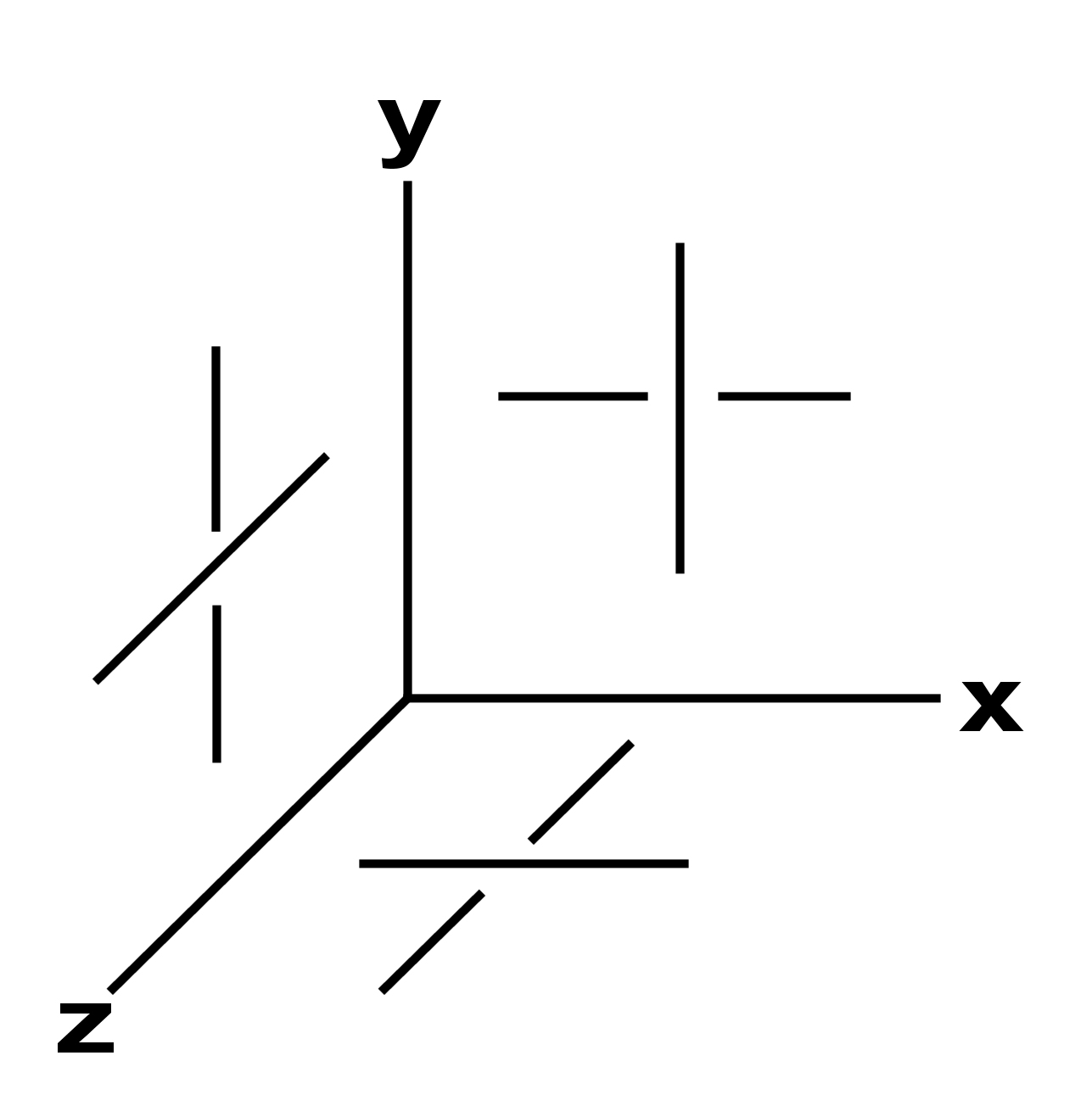}
\caption{Crossing conditions of the knot at every intersection in each projection.}
\label{fig:crossings}
\end{figure}
\end{center}

Arrange the markings in $\Gamma$ so that at every intersection point of the $(x,y)$-projection (i.e., $\pi_z : \mathbb{R}^3 \rightarrow \mathbb{R}^3$ given by $\pi_z(x,y,z)=(x,y)$), the segment parallel to the $x$-axis has smaller $z$-coordinate than the segment parallel to the $y$-axis.  Similarly, arrange so that in the $(y,z)$-projection, $z$-parallel segments cross over the $y$-parallel segments, and in the $(z,x)$-projection, the $x$-parallel segments cross over the $z$-parallel segments (see Figure \ref{fig:crossings}).

A set of markings in $\Gamma$ satisfying the marking conditions and crossing conditions is called a \emph{cube diagram} for the knot or link.

\bigskip
\section{Liftability of grid diagrams}

Because cube diagrams project to grid diagrams, it is natural to think of a cube diagram as a lift of a grid diagram corresponding to, say, the $(x,y)$-projection of the cube.  However, such lifts do not always exist (c.f. \cite{Adam} and \cite{BaldMcCar}).  

Before proceeding, we need to establish some terminology and facts about grid diagrams (for more details see \cite{Adam}).  Grid diagrams will be oriented so that vertical edges are directed from $X$ to $O$.  A \emph{bend} in a grid diagram, $G$, is a pair of segments that meet at a common $X$ or $O$ marking.  We will refer to the former pair of segments as an $X$-bend and the latter as an $O$-bend.  There are two ways to decompose any link component in $G$ into a set of non-overlapping bends, corresponding to a choice of $X$-bends or $O$-bends.  In particular, for a knot there are only two ways to decompose $G$ into non-overlapping bends, and such a choice will be called a \emph{bend decomposition}. 

Consider a grid diagram, $G$, together with a choice of a bend decomposition.  If possible we wish to lift $G$ to a cube diagram where $G$ is the $(x,y)$-projection of the cube diagram and the bend decomposition of $G$ determines the $z$-cube bends of the cube diagram.  While $G$ carries with it an orientation on the knot, so does the $(x,y)$-projection of the cube diagram.  In order that these orientations agree, the $X$-bend decompositon of $G$ has to be chosen--$O$-bends cannot be lifted to $z$-cube bends.  Furthermore, because of the symmetry between all three projections in a cube diagram, it is enough to work only with the $(x,y)$-projection and lift $X$-bends to $z$-cube bends.  

The crossings in a grid diagram sometimes generate a \emph{partial order} on the $X$-bends.  Let $b_1$ and $b_2$ be two $X$-bends.  If $b_1$ crosses over $b_2$ in $G$ we say that $b_1 > b_2$.  Thus in any lift of $G$, the $z$-cube bend corresponding to $b_1$ must have $z$-coordinate greater than that of the $z$-cube bend corresponding to $b_2$.  

Of course, not every grid diagram has a partial order on the $X$-bends.  A grid diagram for which there is no partial order on the $X$-bends may not even lift to a lattice knot that has well-defined knot projections in the other planes (Figure 5 of \cite{Adam}).  However, if there is a partial ordering on the $X$-bends of the grid diagram, it will lift to a lattice knot in which all projections are well-defined knot projections (c.f. \cite{Adam}).  Nevertheless, even a partial order doesn't guarantee liftability to a cube diagram as the $(y,z)$- and $(z,x)$-projections may not be grid diagrams in such a lift (c.f \cite{Adam} and \cite{BaldMcCar}).  Below, we will introduce some grid configurations that fail to lift, not because of a lack of partial ordering but due to crossings in the $(y,z)$- or $(z,x)$-projections that do not satisfy the crossing conditions for a cube diagram.  In Figures \ref{fig:type1} and \ref{fig:type2}, the shaded regions are determined by the corresponding $X$-bend and extend from the $X$-bend to the boundary of the grid diagram as indicated.  Furthermore, a dotted edge represents a sequence of edges in the grid that remains in the shaded region.  This condition is required in the proofs that follow to guarantee that the $(y,z)$- or $(z,x)$-projections contain crossings that do not satisfy the crossing conditions for a cube diagram.

\begin{theorem}
The Type 1 configurations shown in Figure \ref{fig:type1} do not appear in the projection of a cube diagram.
\end{theorem}

\begin{center}
\begin{figure}[h]
\centering
\includegraphics[scale=.35]{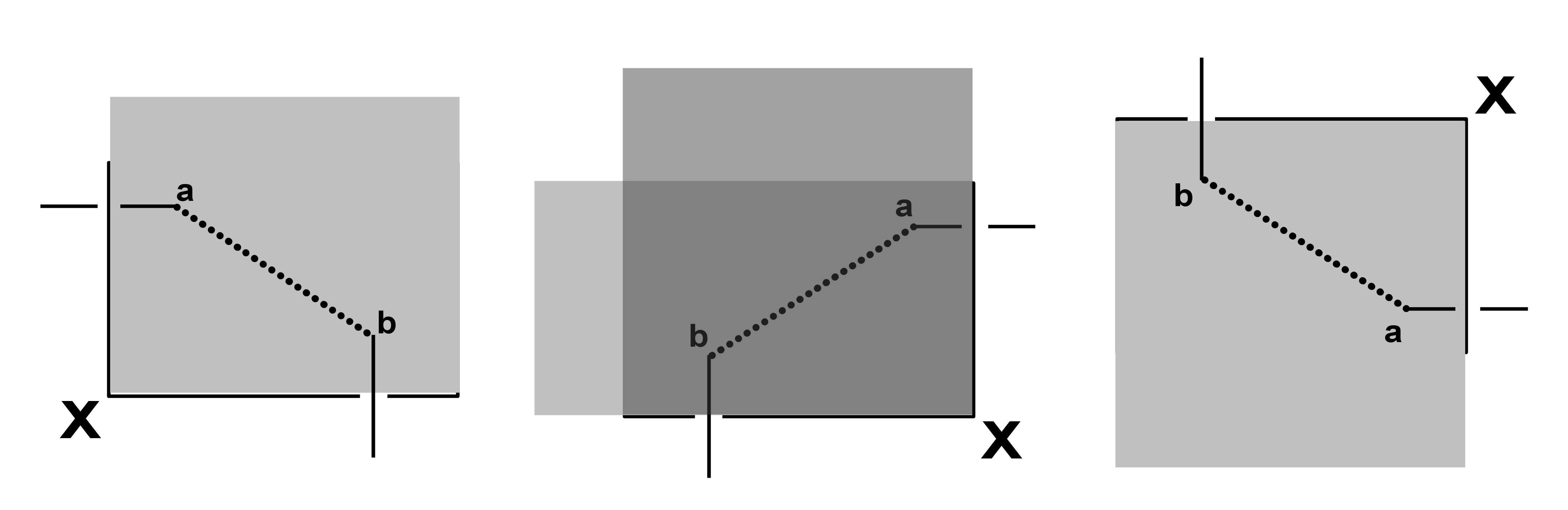}
\caption{Type 1 configurations.}
\label{fig:type1}
\end{figure}
\end{center}

\begin{proof}
We will prove the result for the center configuration.  The remaining cases are similar.  If we assume first that there is a partial order on the $X$-bends of the grid diagram, then the shaded region must contain an $O$ marking.  If there is no $O$ marking in the shaded region, then $a$ and $b$ must end at the same $X$ mark, and thus there is no partial order on the $X$-bends, and no lift.  In any lift of the grid to a lattice knot satisfying the marking conditions for a cube diagram at least one such $O$ marking must represent a vertical edge that passes through the flat containing the $X$-bend shown, since $a$ is below the bend and $b$ is above.  Since this $z$-parallel edge is located in the shaded region, it must either cross over the $x$-parallel edge shown in the $(z,x)$-projection or behind the $y$-parallel edge shown in the $(y,z)$-projection, which breaks the crossing condition shown in Figure \ref{fig:crossings}.
\end{proof}

\begin{center}
\begin{figure}[h]
\centering
\includegraphics[scale=.35]{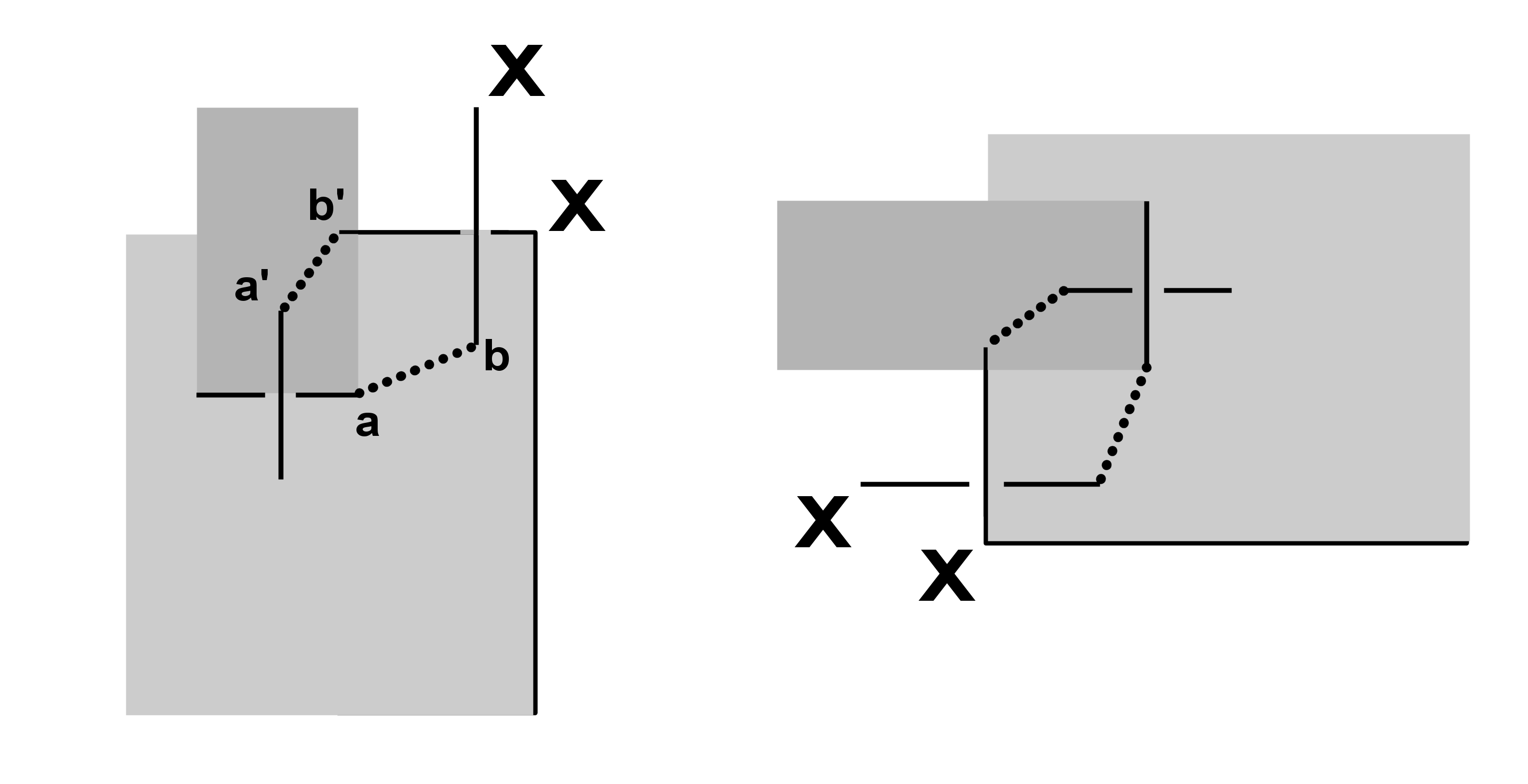}
\caption{Type 2 configurations.}
\label{fig:type2}
\end{figure}
\end{center}

\begin{theorem}
The Type 2 configurations shown in Figure \ref{fig:type2} do not appear in the projection of a cube diagram.
\end{theorem}

\begin{proof}
For the first configuration shown in Figure \ref{fig:type2}, if the edge ending in $a$ has $z$-coordinate less than the lower $X$-bend then the sequence of edges connecting $a$ to $b$ must contain an $O$-marking.  That $O$-marking corresponds to a $z$-parallel edge that crosses behind that lower $X$-bend in the $(y,z)$-projection.  Otherwise, the edge ending in $a$ is above that lower $X$-bend.  In this case, the sequence of edges connecting $a'$ to $b'$ must contain an $O$-marking that corresponds to a $z$-parallel edge that crosses above the edge ending in $a$ in the $(z,x)$-projection.  A similar argument will show that the other configuration fails to lift as well.
\end{proof}

\section{Cube number and chirality}

\begin{theorem}
\label{thm:chirality}
The cube number detects the chirality of the trefoil.  That is, if $K_L$ is the left hand trefoil and $K_R$ is the right hand trefoil, then $c(K_L) < c(K_R)$.
\end{theorem}

\begin{center}
\begin{figure}[h]
\centering
\includegraphics[scale=1.8]{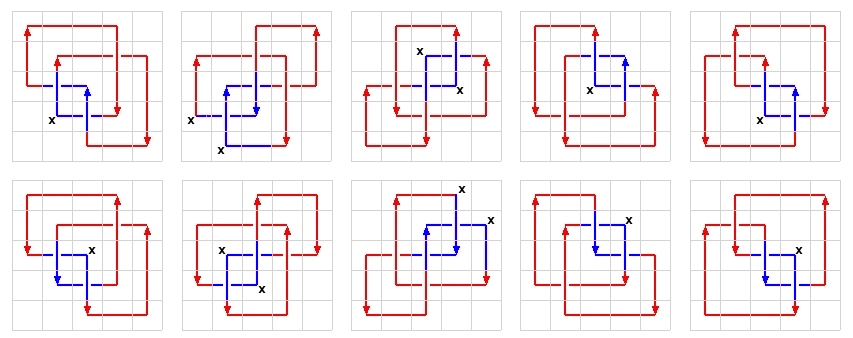}
\caption{Minimal right hand trefoil grid diagrams.}
\label{fig:RHTrefoils}
\end{figure}
\end{center}

\begin{remark}
The cube number of the left hand trefoil is 5 and the cube number of the right hand trefoil is 7.
\end{remark}

\begin{proof}
Note first that $c(K_L) = \alpha(K_L)$ (see Figure \ref{fig:LHTrefoil}).  For the right hand trefoil, Figure \ref{fig:RHTrefoils} shows all minimal grid diagrams.  For columns $1$, $4$, and $5$ there is a Type $1$ configuration present in the grid (shown in blue and marked by the $X$-bend).  For columns $2$ and $3$ there is either no partial order on the bends or there is a Type $2$ configuration (shown in blue and marked by the $X$-bend).  
\end{proof}

For all examples that have been computed, the cube number detects the chirality of the knot as in Theorem \ref{thm:chirality}.  For knots with arc index greater than $5$ proofs of this nature become infeasible.  However, a computer can do the same basic checks for Type 1 and Type 2 configurations.  A program was written that generates all grid diagrams, sifts out those that contain Type 1 and 2 configurations, and then attempts to lift the remaining diagrams to cube diagrams.  Upon finding a valid cube diagram, the Jones polynomial is computed to identify the knot type.  This program has been successfully run up to size $9$ diagrams, generating the following result:

\begin{theorem} 
\label{thm:mainchirality}
The cube number detects the chirality of the following knots: $5_2$, $7_2$, $7_3$, $7_4$, $7_5$, $8_{19}$, $9_{49}$, $10_{124}$, $10_{128}$, $10_{139}$, $10_{145}$, $10_{161}$, $12_{n591}$,  and the $(3,2)$, $(5,2)$, and $(7,2)$ torus knots.
\end{theorem}

For most knots mentioned in Theorem \ref{thm:mainchirality} one knot has cube number equal to arc index and all that is known is that the mirror image has cube number strictly greater than arc index.  However, the calculation has also yielded the result that the cube number for the right hand trefoil is equal to $7$ and the cube number for the right hand $(5,2)$ torus knot is $10$ (the program found no diagrams for the right hand $(5,2)$ torus knot of size $9$ or less and an example of size $10$ exists).

\section{Cube number and Legendrian type}
\label{section:legendrian}

Any grid diagram represents the front projection of a Legendrian knot by following this procedure.  First smooth the northeast and southwest corners.  Then convert northwest and southeast corners to cusps and rotate the grid diagram $45$ degrees counterclockwise.  Note that some authors (see \cite{legend}) describe a similar procedure to obtain a Legendrian projection for the mirror image of the knot represented by $G$, while the one described here is for the original knot type represented by $G$.  While there is no similar construction to convert a cube diagram into a Legendrian knot, each of the projections of a cube diagram is a grid diagram, and hence, represents the Legendrian front projection of some knot.  Therefore one can define the Legendrian cube number, $c_\ell(K)$, to be the smallest $n$ such that there is a cube diagram for the knot $K$ of size $n$ where the $(x,y)$-projection of the cube diagram is a grid diagram representing the Legendrian knot $K$.  

According to \cite{etnyrehonda} the \emph{rotation number} may be computed from the Legendrian front projection as follows:
$$r(K) = \frac{1}{2}(D_c - U_c)$$
where $D_c$ is the number of downward oriented cusps and $U_c$ is the number of upward oriented cusps in the Legendrian front projection.  Furthermore, according to \cite{etnyrehonda} and \cite{Ng}, any minimal grid diagram for a left hand Legendrian torus knot, $T_{p,2}$, must realize the maximal Thurston Bennequin number and hence have rotation number satisfying:
$$r(K) \in \left\{ \pm (|p| - 2 - 4t): t \in \mathbb{Z}, 0 \leq t < \frac{|p| - 2}{2} \right\}.$$
For the following results we will need to refer to a \emph{standard diagram}, $G_{j,k}$, for a $(p,2)$-torus knot.  A standard diagram for $T_{p,2}$ is one that contains all crossings along two diagonals as shown in Figure \ref{fig:standardp2} (note that there is a similar notion of standard diagram for right hand $(p,2)$-torus knots).  For the standard diagram shown in Figure \ref{fig:standardp2} the rotation number may be computed as follows:  
$$r(G_{j,k}) = \frac{1}{2} ( 2k - 2j ) = k - j.$$
Note that each Legendrian left hand $(p,2)$-torus knot realizable in a minimal grid diagram is represented by a standard grid diagram.  Also the standard diagram with maximum rotation number is one in which $k = p - 1$ (recall that $p = k + j$).  Similarly the standard diagram with minimum rotation number has $j = p - 1$.

\begin{center}
\begin{figure}[h]
\centering
\includegraphics[scale=.5]{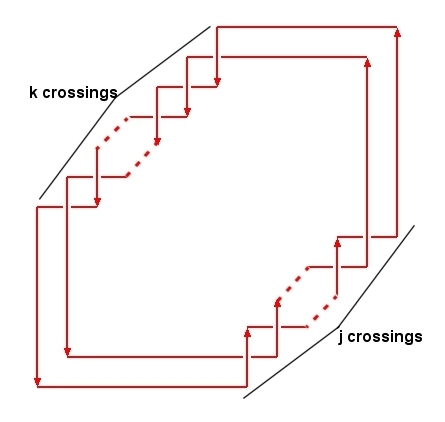}
\caption{Standard diagram $G_{j,k}$ where $j+k = p$ and $j,k \geq 1$.}
\label{fig:standardp2}
\end{figure}
\end{center}

\begin{theorem}
\label{thm:legendrian}
Let $K_{min}$ be the left hand $(5,2)$-torus knot with maximal Thurston-Bennequin number and $r = -3$ and $K_{max}$ the $(5,2)$-torus knot with maximal Thurston-Bennequin number and $r = 3$.  Then the Legendrian cube number distinguishes between $K_{min}$ and $K_{max}$.
\end{theorem}

\begin{proof}
Note first that $c_\ell(K_{max}) = \alpha(T_{5,2})$ (the construction shown in Figure \ref{fig:LHTrefoil} extends to $T_{5,2}$).  Now, for $K_{min}$ all minimal grid diagrams can be reached from the standard one (see Figure \ref{fig:lhp2}) by cyclic permutation.  This fact has been verified by a computer program that determines the total number of grid diagrams of a given Legendrian knot.  In this case the program confirmed that there are only $49$ distinct minimal diagrams for $K_{min}$, precisely those that may be reached from the standard diagram by cyclic permutation.  Finally, cyclic permutations do not eliminate all of the Type 1 configurations present in each of the $49$ diagrams.  Therefore $c_\ell(K_{min}) > \alpha(T_{5,2})$.
\end{proof}

There are also further results regarding cube number and Legendrian knot type.  A computer program has been used to show that cube number distinguishes between the Legendrian $(7,2)$ torus knots with $tb = 0$ and $r = \pm 5$.  The following conjecture has been verified up to $p = 7$ by a computer program.  

\begin{conjecture}
\label{gridMoves}
Given a torus knot $T_{p,2}$ and a legendrian knot $K_L$ of topological type $T_{p,2}$ and a grid diagram for $K_L$ of size $\alpha(T_{p,2})$.  Then there is a sequence of commutation and cyclic permutation moves taking $K_L$ to the standard grid diagram representing it.
\end{conjecture}

To understand the conjecture, we must first define what is meant by commutation and cyclic permutation.  A cyclic permutation of a grid diagram takes a column (resp. row) adjacent to the boundary of the grid, and moves it to the opposite end of the grid (c.f. \cite{Cromwell}).  Two adjacent columns (resp. rows) are said to be \emph{interleaved} if the $y$-coordinates (resp. $x$-coordinates) of the markings are distinct and alternate between the columns (resp. rows).  A commutation switches the position of any two adjacent rows (resp. columns) that are not interleaved, including pairs that are adjacent with respect to the boundary (c.f. \cite{Adam}).  Note that the grid may be thought of as existing on a torus, so opposite edges may be identified (hence the commutation move includes rows and columns adjacent to the boundaries).  Furthermore, using this definition of commutation, commutation and cyclic permutation commute.

\begin{conjecture}
Let $K_{max}$ and $K_{min}$ be left hand Legendrian torus knots with $p \geq 5$ and maximal Thurston-Bennequin number such that the rotation number of $K_{max}$ is maximal and the rotation number of $K_{min}$ is minimal.  Then $c_\ell(K_{max}) < c_\ell(K_{min})$.
\end{conjecture}

\begin{proof}[Proof (Assuming conjecture \ref{gridMoves}).]
As in the proof of Theorem \ref{thm:legendrian} the cube diagram shown in Figure \ref{fig:LHTrefoil} extends to give a cube diagram for $K_{max}$ of size equal to the arc index.  Also for $p = 5$ this result is precisely Theorem \ref{thm:legendrian}.  A similar argument shows that the result holds for $p = 7$ as well.  Assume now that $p \geq 9$.  Since cyclic permutation commutes with commutation, all commutation moves may be moved to the beginning of the sequence. Since the standard grid diagram for $K_{min}$ has no possible commutation moves, it suffices to consider only cyclic permutations.  Consider a sequence of cyclic permutation moves.  Such a sequence is equivalent one in which the first $a$ moves are cyclic permutations down (i.e. moving the bottom row to the top of the grid) and the remaining moves, denote the number of moves by $b$, are cyclic permutations to the left (i.e. moving the left column to the right side of the grid).  Also, note that doing $a$ cyclic permutations down is equivalent to doing $n - a$ cyclic permutations up.  Similarly, doing $b$ cyclic permutations left is equivalent to doing $n - b$ cyclic permutations right.  

\begin{center}
\begin{figure}[h]
\centering
\includegraphics[scale=.5]{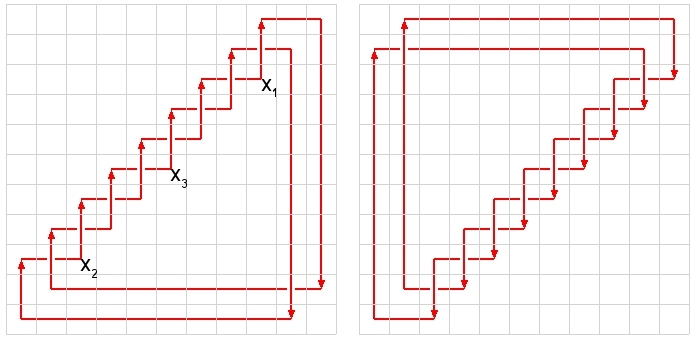}
\caption{Standard diagrams for the $(p,2)$ torus knot with $r = -(p - 2)$ and $r = p - 2$ respectively.}
\label{fig:lhp2}
\end{figure}
\end{center}

Now, if $a \leq n - 4$ and $b \leq n - 5$ the upper right most Type 1 configuration (labeled $X_1$ in Figure \ref{fig:lhp2}) will remain fixed.  If $a \geq 5$ and $b \geq 4$ then the lower left most Type 1 configuration (labeled $X_2$ in Figure \ref{fig:lhp2}) remains.  If $a > n - 4$ and $b < 4$, or if $a < 5$ and $b > n - 5$ then the Type 1 configuration labeled $X_3$ remains.


\end{proof}

Note that while chirality is detected in all examples computed so far, it is not true that cube number always detects legendrian knot type.  There are multiple examples of left hand $(p,2)$ torus knots that are of different Legendrian types yet still have cube number equal to arc index.  For example, when $p = 7$ all left hand Legendrian $(p,2)$ torus knots realizable in a minimal grid diagram have cube number equal to arc index except for those with rotation number equal to $-3$ and $-5$.

\end{document}